\documentclass[11pt,reqno]{amsart}  
\usepackage{graphicx}
\usepackage{amssymb}
\usepackage{epstopdf}
\usepackage{romannum}
\usepackage{color}
\usepackage{enumitem}

\usepackage[colorlinks = true,
            linkcolor = black,
            urlcolor  = blue,
            citecolor = black,
            anchorcolor = black
            ]{hyperref}

\usepackage{mathtools}
\usepackage{hyperref}
\usepackage{cleveref}
\usepackage[foot]{amsaddr}
\usepackage{makecell, caption, booktabs, multirow, siunitx}

\usepackage{array}

\newtheorem{theorem}{Theorem}
\newtheorem{lemma}[theorem]{Lemma}
\newtheorem{proposition}[theorem]{Proposition}
\newtheorem{corollary}[theorem]{Corollary}

\theoremstyle{definition}
\newtheorem{definition}[theorem]{Definition}

\newtheorem{example}[theorem]{Example}
\theoremstyle{remark}

\usepackage[textsize=footnotesize]{todonotes}

\title{Approximate Pythagoras Numbers on $*$-algebras over $\mathbb{C}$}
\author{Paria Abbasi}
\address{Department of Mathematics, University of Innsbruck, Austria}
\author{Sander Gribling}
\address{IRIF, Universit\'e de Paris, France}
\author{Andreas Klingler}
\address{Institute for Theoretical Physics, University of Innsbruck, Austria}
\author{Tim Netzer}
\address{Department of Mathematics, University of Innsbruck, Austria}

\date{\today}                                           

\begin{document}

%

\begin{abstract}
The Pythagoras number of a sum of squares is the shortest length among its sums of squares representations. In many algebras, for example real polynomial algebras in two or more variables, there exists no upper bound on the Pythagoras number for all sums of squares.
In this paper, we study how Pythagoras numbers in $*$-algebras over $\mathbb{C}$ behave with respect to small perturbations of elements. More precisely, the approximate Pythagoras number of an element is the smallest Pythagoras number among all elements in its $\varepsilon$-ball. We show that these approximate Pythagoras numbers are often significantly smaller than their exact versions, and allow for (almost) dimension-independent upper bounds. Our results use low-rank approximations for Gram matrices of sums of squares and estimates for the operator norm of the Gram map.
\end{abstract}

\maketitle

\section{Introduction and Preliminaries}


Artin's solution to Hilbert's 17th Problem states that every nonnegative rational function $f$ in $\mathbb{R}(x_1, \ldots, x_n)$ admits a decomposition into a sum of squares of rational functions. But how many squares are necessary to decompose an element? This is captured by its Pythagoras number, i.e. the minimum integer $r$ such that there exists a decomposition
$$ f = \sum_{k=1}^r p_k^2.$$
For real rational functions, Pfister \cite{pf67} showed that the Pythagoras number of all sums of squares is upper bounded by $2^n$, but it is still unknown how sharp this bound is in general. 

The situation is very different for polynomial rings. Although for univariate polynomials every sum of squares is still a sum of two squares, for two or more variables there exists no upper bound for the Pythagoras numbers of all elements \cite{ch82}.

Sums of squares of real polynomials have recently attained great interest from the point of view of optimization  (see \cite{MR4249542} for an overview). Polynomial optimization problems can be approached via semidefinite programming by relating sums-of-squares decompositions with positive semidefinite Gram matrices.  The Pythagoras number here corresponds to the smallest rank of a positive semidefinite Gram matrix. Since almost all present algorithms for semidefinite optimization are numerical, a certain error in the results will usually occur and be accepted. This led us to the question of how Pythagoras numbers behave with regard to small errors/perturbations of the polynomials. This is closely related to low-rank approximations of Gram matrices (and the operator norm of the Gram map).

The concept of sum-of-squares decompositions and their corresponding Pythagoras numbers is not exclusive to polynomials. Given any complex $*$-algebra, we call an element $a$ a sum of (Hermitian) squares if it admits a decomposition
$$a = a_1^{*} a_1 + \cdots +a_m^* a_m,$$ and the smallest such $m$ is its Pythagoras number.

In this paper, we introduce and study  approximate Pythagoras numbers in this broader context.\footnote{The choice of the ground field $\mathbb{C}$ excludes the classic case of $\mathbb{Z}$ as in the famous Four-Square Theorem of Lagrange. But it is essential to talk about small perturbations.} Our main findings can be summarized by saying that approximate Pythagoras numbers are often significantly smaller than the known upper bounds for the exact case. 

Let us emphasize here that our framework requires to \emph{fix} some small approximation error. When letting this error go to zero, for example, when computing Border ranks of matrices or tensors, the results will not be meaningful anymore. However, numerical computations often work with a fixed precision, which is then covered by our approach.

Our approach further allows computing the approximate numbers efficiently (involving semidefinite programming) in cases where the exact Pytha\-goras numbers are very hard to compute, or even unknown.
Moreover, we will apply these results to concrete algebras of commutative and non-commutative complex polynomials, and compute explicit upper bounds of the approximate Pythagoras numbers for sums of Hermitian squares. This also includes real polynomial rings, since any Hermitian square is just the same as a sum of two classical squares here.


This paper is structured as follows.  In \Cref{ss:ph} we introduce Pythagoras numbers and their approximate versions. In \Cref{ss:gram} we explain the Gram map and some of its consequences. \Cref{s:main} contains our main result including an upper bound for the approximate Pythagoras number in general $*$-algebras over $\mathbb{C}$. In particular, we prove in \Cref{ss:sos} an approximation result for positive semidefinite matrices and deduce the main result for approximate Pythagoras numbers (\Cref{mainthm}). In \Cref{ss:free} we apply the results to non-commutative polynomials, and in \Cref{ss:poly} to commutative polynomials. The final \Cref{ss:norm} includes some considerations on the so-called sos-norm that appears as a relevant measure throughout our paper. \Cref{app} contains a table summarizing the notations and definitions of all norms we use throughout the paper.

\subsection{Pythagoras Numbers}\label{ss:ph}
\begin{definition}
Let $\mathcal A$ be a  $*$-algebra over $\mathbb C$. For any sum of Hermitian squares $a\in\sum\mathcal A^2$  the  {\it Pythagoras number} is defined  as $$\mathfrak p(a)\coloneqq \min\left\{m\in\mathbb N\mid \exists a_1,\ldots, a_m\in\mathcal A\colon a=a_1^*a_1+\cdots +a_m^*a_m \right\}$$ i.e.\ as the shortest length of a sum of squares representation of $a$ in $\mathcal A$. For any subset $\mathcal S\subseteq \mathcal A$ we define its  Phythagoras number as $$\mathfrak p(\mathcal S)\coloneqq {\rm sup}\left\{ \mathfrak p(a)\mid a\in \mathcal S\cap\sum\mathcal A^2\right\}.$$ 
\end{definition}

Note that the Pythagoras number of a single element is also called \emph{sum of squares length} in the literature (see for example \cite{sch16}). To use only one name for similar concepts, we will call both $\mathfrak{p}(a)$ and $\mathfrak{p}(\mathcal{S})$ a Pythagoras number.

\begin{example} ($i$) In \cite{pf67} it was shown that every sum of squares in the field $\mathbb R(x_1,\ldots, x_n)$ is a sum of  at most $2^n$ squares. This bound is tight for $n \leqslant 2$ but it is still unclear whether it can be improved for $n \geqslant 3$. The only known lower bound on this number is $n+2$ (see for example \cite{sch16} for an overview). 

 In our setting we consider the algebra $\mathcal A=\mathbb C(x_1,\ldots, x_n)$ with involution acting as complex conjugation on the coefficients. Then a Hermitian square in $\mathcal A$ is the same as a sum of two classical squares in $\mathbb R(x_1,\ldots, x_n)$, so  in our notation we have $$\frac{n}{2}+1 \leqslant \mathfrak p(\mathbb C(x_1,\ldots, x_n))\leqslant 2^{n-1}.$$

($ii$) In the commutative polynomial algebra $\mathcal A=\mathbb C[x_1,\ldots, x_n],$ again with involution acting as conjugation on coefficients, we have $\mathfrak p(\mathbb C[x_1])=1$  and $$\mathfrak p(\mathbb C[x_1,\ldots, x_n])=\infty$$ for $n\geqslant 2$. For $n=1$ this is easy to see from the decomposition of a nonnegative polynomial into irreducible factors (we again get the factor of $1/2$ when compared to the usual way of stating the result for real polynomials, since we use the complex Hermitian setup), the result for  $n\geqslant 2$ was proven in \cite{ch82}.

($iii$) The situation becomes  more involved  when considering polynomials of fixed degree, for example when computing $$\mathfrak p\left(\mathbb C[x_1,\ldots, x_n]_{2d}\right)$$ where $\mathbb C[x_1,\ldots, x_n]_{2d}$ denotes the  subspace   of homogeneous polynomials (a.k.a.\ forms) of degree $2d$. We do not state the known results in detail, but refer to \cite{sch16} instead.  For a few small cases of $n$ and $d$ the value  is known exactly, and in the general case an upper bound is known. The growth of this upper  bound is  $$\mathcal O\left(d^{\frac{n-1}{2}}\right)$$ and  in \cite{sch16} it is proven (up to a conjecture of Iarrobino-Kanev from algebraic geometry)  that this is  asymptotically tight as $d\to\infty.$ 

($iv$) For  the free polynomial algebra $\mathcal A=\mathbb C\langle z_1,\ldots, z_n\rangle$ with involution defined by $z_i^*=z_i,$ we also get $\mathfrak p(\mathcal A)=\infty$ for $n\geqslant 2$. This can either be shown directly (we will see this below), but it also follows immediately from the commutative result. Indeed every sum of squares in $\mathbb C[x_1,\ldots, x_n]$ admits a sum of squares preimage  in $\mathbb C\langle z_1,\ldots, z_n\rangle$   (w.r.t\ the commutative collapse map $\mathfrak c\colon \mathbb C \langle z_1,\ldots, z_n\rangle\to \mathbb C[x_1,\ldots, x_n]$), whose Pythagoras number is as large as the inital one.
\end{example}

We now turn to {\it approximate versions} of Pythagoras numbers. We equip the algebra $\mathcal A$ (or at least some subspace) with a norm $\Vert\cdot\Vert_{\mathcal{A}}$. Given $a\in\sum\mathcal A^2$, we study the smallest Pythagoras number realized by elements in the $\varepsilon$-ball around $a$.  

Some words of warning are required here.  This  approximate Pythagoras number clearly depends on $\varepsilon$, but might also  depend significantly on the subspace in which the approximating sum of squares is contained, or even on the subspace from which the elements to be squared are chosen. While the classical Pythagoras number does not change under positive scaling, the approximate version clearly does, and thus needs to involve some measure of the size of $a$. It thus does not make  sense to define it for a full subspace.

Not to overload notation, we thus refrain from introducing some symbol to denote the $\varepsilon$-Pythagoras number of an element or even a set of elements, but state all results about approximations as explicitly as possible.
 However, the main principle we will see in the following is the following: {\it Approximate Pythagoras numbers are often significantly smaller than  the exact  Pythagoras numbers or the previously known upper bounds.}

\subsection{The Gram Map}\label{ss:gram} An essential ingredient for the below results is the Gram map, which relates sums of squares to positive semidefinite matrices. It was introduced in \cite{ChLR95} and has been used widely since. 

For the rest of the paper, let ${\rm Mat}_d(\mathbb C)$ be the set of complex $d \times d$ matrices and ${\rm Psd}_d(\mathbb C) \subseteq {\rm Mat}_d(\mathbb C)$ the convex cone of positive semidefinite matrices. Further, let $\mathcal A$ be a unital $*$-algebra over $\mathbb C$. We fix a linear subspace $\mathcal V\subseteq \mathcal A$, together with a basis $\underline v=(v_1,\ldots, v_d)$. Then the associated {\it Gram map} is 
\begin{align*}
G_{\underline v} \colon {\rm Mat}_d(\mathbb C) & \to  \mathcal A \\ M=(m_{ij})_{i,j} &\mapsto (v_1^*,\ldots, v_d^*) M (v_1,\ldots, v_d)^t=\sum_{i,j=1}^d m_{ij}v_i^*v_j.
\end{align*}
It is easy to see that $G_{\underline v}$ is a $*$-linear map, whose image is $$\mathcal V^*\mathcal V\coloneqq {\rm span}\left\{ v_i^*v_j\mid i,j=1,\ldots, d\right\}.$$
Moreover, the cone $$\sum \mathcal V^2\coloneqq \left\{ \sum_{k=1}^m w_k^*w_k\mid m\in\mathbb N, w_k\in \mathcal V\right\}$$ of sums of Hermitian squares  of elements from $\mathcal V$ coincides with the image $$G_{\underline v}({\rm Psd}_d(\mathbb C)).$$ This follows directly from the following observation: Given $w_1,\ldots, w_m\in\mathcal V$, express each  $w_k$ in the basis $\underline{v}$,
$$ w_k = \sum_{i=1}^d c_{ki} v_i,$$
and consider the matrix $M = (m_{ij})_{i,j} \in {\rm Psd}_d(\mathbb C)$ defined by  
$$m_{ij} \coloneqq \sum_{k = 1}^{m} \overline c_{ki} \cdot c^{\phantom{*}}_{kj}.$$
We immediately obtain $G_{\underline{v}}(M) = \sum_{k=1}^m w_k^*w_k$.
Conversely, each  positive semidefinite matrix $M$ is a sum of ${\rm rank}(M)$ many Hermitian squares of rank one in ${\rm Mat}_d(\mathbb C)$. Then $G_{\underline v}(M)$  is a sum of at most ${\rm rank}(M)$ many Hermitian squares from $\mathcal V$.

This observation also implies that the Pythagoras number of an element equals the minimal rank among all of its positive semidefinite Gram matrices, which further implies $$\mathfrak p\left(\sum\mathcal 
V^2\right)\leqslant {\rm dim}_{\mathbb C}\left(\mathcal V\right).$$
With the following lemma (see \cite{Bar02} \Romannum{2}.\ 14, Problem 4), one can lower this general upper bound in certain cases. 
\begin{lemma}\label{lmm:1}
Let $A_1,\ldots, A_k \in {\rm Mat}_d(\mathbb{C})$ be Hermitian and $\alpha_1,\ldots,\alpha_k\in \mathbb{R}$.

If the system of equations ${\rm tr}( A_i X)=\alpha_i$ for $i=1,\ldots, k$ has a positive semidefinite solution $X \in {\rm Psd}_d(\mathbb C)$, then there is also a positive semidefinite solution $X_0\in {\rm Psd}_d(\mathbb C)$ with ${\rm rank}(X_0)\leqslant r$,
whenever $r$ satisfies $k\leqslant r^2+2r$. 
\end{lemma}

\begin{corollary}\label{cor:impcar} For every finite-dimensional subspace $\mathcal V\subseteq\mathcal A$ we have 
$$\mathfrak p\left(\sum\mathcal 
V^2\right)\leqslant \left\lceil \sqrt{\dim(\mathcal V^*\mathcal V)}\right\rceil.$$ 
\end{corollary}
\begin{proof}Let $\underline v=(v_1,\ldots, v_d)$ be a basis of $\mathcal V$, set $k\coloneqq \dim\left(\mathcal V^*\mathcal V\right),$ and equip the vector space $\mathcal V^*\mathcal V$ with a basis $\underline w=(\omega_1, \ldots, \omega_k)$ such that $\omega_{i}=\omega_{i}^*$ for all $i=1,\ldots, k$. For $a \in \sum\mathcal 
V^2$ there exists a positive semidefinite matrix $M \in {\rm Psd}_d(\mathbb C)$ such that 
$$a=G_{\underline v}(M)=\sum_{i,j=1}^{d}m_{ij}v_i^*v_j.$$ 
Write $v_i^*v_j=\sum_{l=1}^{k}\lambda_{l}^{i,j}\omega_l$ and $a=\sum_{l=1}^{k}\lambda_{l}\omega_{l}$ with all $\lambda_{l}^{i,j},\lambda_l\in \mathbb{C}$. We  then have
$$\sum_{i,j=1}^{d}m_{ij}v_i^*v_j=\sum_{l=1}^{k}\left(\sum_{i,j=1}^{d}m_{ij}\lambda_{l}^{i,j}\right)\omega_l=\sum_{l=1}^{k}\lambda_{l}\omega_{l},$$
which implies 
$$\sum_{i,j=1}^{d}m_{ij}\lambda_{l}^{i,j}=\lambda_{l} \quad \mbox{ for } l=1,\ldots, k.$$
Consider the Hermitian matrices $$A_{l}\coloneqq \left(\overline\lambda_{l}^{i,j}\right)_{i,j}\in{\rm Mat}_d(\mathbb C)$$ for $l=1,\ldots, k$. The  positive semidefinite matrix $M$ is then clearly a solution of the following equations:
$$\label{semidef}{\rm tr}(A_lM)=\lambda_{l} \quad \mbox{ for } l=1,\ldots, k.$$  Now $r\coloneqq  \lceil \sqrt{k}\rceil$ satisfies the requirements of Lemma \ref{lmm:1}, and thus there also exists a positive semidefinite solution $M_0$ with ${\rm rank}(M_0)\leqslant \lceil \sqrt{k}\rceil.$ Applying the Gram map to $M_0$ shows that $a$ is a sum of at most $r$ many squares.
\end{proof}

Note that $\sqrt{\dim\left(\mathcal V^*\mathcal V\right)}\leqslant \dim(\mathcal V)$, but the inequality can be strict, in which case Corollary \ref{cor:impcar} might give a better upper bound for $\mathfrak p\left(\sum\mathcal 
V^2\right)$ than $\dim(\mathcal V).$

\begin{example}\label{ex:spacedim}
For the  space $\mathcal{V}=\mathbb C[x_1,\ldots, x_n]_{d}$ of homogeneous polynomials of degree  $d$ we have 
$\mathcal V^*\mathcal V=\mathbb C[x_1,\ldots, x_n]_{2d}$ with 
$${\rm dim}_{\mathbb C}\left(\mathcal V^*\mathcal{V}\right)={{2d+n-1}\choose{n-1}}.$$ This  gives  an upper bound for the Pythagoras number $\mathfrak p\left(\sum \mathcal V^2\right)$ that grows like  $\mathcal O\left(d^{\frac{n-1}{2}}\right)$ for $d\to\infty$.\end{example}


\section{Main Results}\label{s:main}
\subsection{Approximating Sums of Squares}\label{ss:sos}
For $1\leqslant p<\infty$ we equip the space ${\rm Mat}_d(\mathbb C)$ with the  the Schatten $p$-norm,
$$\Vert M\Vert_p \coloneqq \left({{\rm tr}(\sqrt{M^*M}^p)}\right)^{1/p},$$
and also use the Schatten $\infty$-norm $\Vert M \Vert_{\infty} \coloneqq  \sigma_{\max}(M)$, which is the largest singular value of $M$.
The following proposition is the main technical ingredient for all of our later results. It uses a well-known technique in low rank approximation theory, we include the proof for completeness.

\begin{proposition}\label{prop:sanderandinfinity}
Let $M\in{\rm Psd}_d(\mathbb C)$ and $\varepsilon >0$ be fixed. 
\begin{enumerate}[label=(\roman*)]
\item For $1 < p < \infty$, there exists a matrix $M' \in {\rm Psd}_d(\mathbb C)$ such that $$\Vert M-M'\Vert_p\leqslant \varepsilon$$ and $${\rm rank}(M')<  \left(\frac{{\rm tr}(M)}{\varepsilon}\right)^{\frac{p}{p-1}}$$

\item There exists a matrix $M' \in {\rm Psd}_d(\mathbb C)$ such that $\Vert M-M'\Vert_\infty\leqslant \varepsilon$ and $${\rm rank}(M')< \frac{{\rm tr}(M)}{\varepsilon}.$$
\end{enumerate}
\end{proposition}
\begin{proof}
Let $$ M = \sum_{i=1}^{d} \lambda_i v_i v_i^{*}$$
be a spectral decomposition of $M$, where we assume $\lambda_1\geqslant \lambda_2\geqslant \cdots \geqslant \lambda_d\geqslant 0$. 

($i$) For each $k=1,\ldots, d$ we have $${\rm tr}(M)=\sum_{i=1}^d\lambda_i \geqslant \sum_{i=1}^k\lambda_i\geqslant k\lambda_k,$$ which implies $\lambda_k\leqslant {\rm tr}(M)/k.$ Choose $k\in\mathbb N$ with $k<\left(\frac{{\rm tr}(M)}{\varepsilon}\right)^{\frac{p}{p-1}}\leqslant k+1,$ and define $$M':=\sum_{i=1}^k\lambda_iv_iv_i^*\in{\rm Psd}_d(\mathbb C)$$ which fulfills  ${\rm rank}(M')\leqslant k< \left(\frac{{\rm tr}(M)}{\varepsilon}\right)^{\frac{p}{p-1}}.$ The following computation thus completes the proof of ($i$):  $$\Vert M-M'\Vert_p^p=\sum_{i=k+1}^d\lambda_i^p\leqslant \lambda_{k+1}^{p-1}\cdot \sum_{i=k+1}^d\lambda_i\leqslant \lambda_{k+1}^{p-1}\cdot {\rm tr}(M)\leqslant \frac{{\rm tr}(M)^p}{(k+1)^{p-1}}\leqslant \varepsilon^p.$$

($ii$)
Choose $k$ maximal with $\lambda_k>\varepsilon$ and define 
$$ M' = \sum_{i=1}^k \lambda_i v_i v_i^{*} \in \textrm{Psd}_d(\mathbb{C}).$$
Then $\Vert M - M' \Vert_{\infty} \leqslant \varepsilon$ and from ${\rm tr}(M)>k\varepsilon$ we obtain 
\begin{equation}\textrm{rank}(M') \leqslant k <  \frac{{\rm tr}(M)}{\varepsilon}.\tag*{\qedhere}\end{equation} \end{proof}

Now we also fix   a vector space norm $\Vert\cdot\Vert_{\mathcal{A}}$ on the algebra $\mathcal A$.
Since the Gram map $$G_{\underline v}\colon {\rm Mat}_d(\mathbb C) \to \mathcal A$$  is linear and defined on a finite-dimensional space, it is bounded, i.e.\ for any constant $1\leqslant p \leqslant  \infty$ we can find some positive constant $C$ such that $$\Vert G_{\underline v}(M)\Vert_{\mathcal A} \leqslant C\cdot \Vert M\Vert_p$$ for all $M\in {\rm Mat}_d(\mathbb C)$. For a fixed choice of $p$, the smallest such $C$ is called the {\it operator norm} of $G_{\underline v}$ and is denoted by $\Vert G_{\underline v}\Vert_{p \to \mathcal{A}}$.

Recall that a sum of squares $a\in\sum\mathcal V^2$ might have several positive semidefinite Gram matrices. Since the trace of a Gram matrix plays a crucial role in the following, this motivates the subsequent definition for elements $a\in\sum\mathcal V^2$: $$\Vert a\Vert_{\underline v,{\rm sos}}\coloneqq {\rm min}\left\{ {\rm tr}(M)\mid M\in {\rm Psd}_d(\mathbb C), G_{\underline v}(M)=a\right\}.$$ Note that the minimum is attained since the set of positive semidefinite Gram matrices of $a$ is closed.  Moreover, for elements in $\sum\mathcal V^2,$  $\Vert \cdot\Vert_{\underline v,\rm sos}$ indeed behaves like a norm, i.e.\ it is positive definite, homogeneous, and satisfies the triangle inequality. 

For example, using trace-minimization to approximate rank minimization is a common technique in systems and control theory. Note that computation of $\Vert a\Vert_{\underline v,\rm sos}$ is a semidefinite program, since it is a minimization of the linear function ${\rm tr}$ along the set of positive semidefinite Gram matrices of $a$.
Nevertheless, computing $\Vert a\Vert_{\underline v, \rm sos}$ analytically from $a$  might not be straightforward. Any explicit sum of squares representation of $a$ gives rise to a positive semidefinite Gram matrix of $a$, whose trace then upper bounds $\Vert a\Vert_{\underline v, \rm sos}$. However, note that not all of the diagonal entries of a Gram matrix might be visible directly from $a$.

\begin{theorem}\label{mainthm}
Let  $a\in\sum\mathcal V^2$ and let $\varepsilon >0$ be fixed. Then the  $\varepsilon$-ball around $a$ contains  a sum of at most  $$ \min_{1<p\leqslant\infty}\left(\frac{\Vert G_{\underline v}\Vert_{p \to \mathcal{A}}\cdot\Vert a\Vert_{\underline v,{\rm sos}}}{\varepsilon}\right)^{\frac{p}{p-1}}$$ many Hermitian squares of elements from $\mathcal V$ (where we use the convention $\frac{\infty}{\infty-1}=1$).
\end{theorem}

\begin{proof}Let $M$ be a positive semidefinite Gram matrix for $a$ with $${\rm tr}(M)=\Vert a\Vert_{\underline v, \rm sos}.$$ For any $1<p\leqslant \infty$ we apply Proposition  \ref{prop:sanderandinfinity} to choose some $M'$ with $$\Vert M-M'\Vert_p\leqslant \varepsilon/\Vert G_{\underline v}\Vert_{p \to \mathcal{A}}$$ and of rank smaller than $$r:=\left(\frac{\Vert G_{\underline v}\Vert_{p\to\mathcal A}\cdot{\rm tr}(M)}{\varepsilon}\right)^{\frac{p}{p-1}}.$$ We now simply apply $G_{\underline v}$ and obtain for $a' \coloneqq G_{\underline v}(M'):$ $$\Vert a-a'\Vert_{\mathcal A} =\Vert G_{\underline v}(M)-G_{\underline v}(M')\Vert=\Vert G_{\underline v}(M-M')\Vert\leqslant \Vert G_{\underline v}\Vert_{p \to \mathcal A}\Vert M-M'\Vert_p\leqslant \varepsilon.$$
Since $M'$ is a sum of at most $r$  many  squares in ${\rm Mat}_d(\mathbb C),$ so is $a'$ in $\sum\mathcal V^2.$
\end{proof}

Note that the last result does not involve the dimension of $\mathcal V$ explicitly; however, it is implicitly still contained in $\Vert G_{\underline v}\Vert_{p \to \mathcal A}$ and also in $\Vert a\Vert_{\underline v, \rm sos}$. While $\Vert G_{\underline v}\Vert_{p \to \mathcal A}$ is constant for some instances when choosing a suitable basis, $\Vert a\Vert_{\underline v, \rm sos}$ cannot be bounded independently of the dimension of $\mathcal V$ in general, as we will see.  Nevertheless, in certain cases such a dimension-independent bound is possible, making Theorem \ref{mainthm} a dimension-independent result on approximate Pythagoras numbers.

\subsection{Non-Commutative Polynomials}\label{ss:free}

 Let $\mathbb C\langle \underline z\rangle$ be the $*$-algebra of  non-commutative polynomials in the Hermitian variables $z_1,\ldots, z_n$.  By $\mathbb C\langle \underline z\rangle_{d}$ we denote the subspace of homogeneous polynomials of degree $d$. We choose the basis $\underline v$ of $\mathbb C\langle \underline z\rangle_{d}$ consisting of all words in the variables $z_1,\ldots, z_n$ of length $d.$ Now note that the corresponding Gram map $$G_{\underline v}\colon {\rm Mat}_{n^d}(\mathbb C)\to \mathbb C\langle \underline z\rangle_{2d}$$ is an isomorphism.  Every word of length $2d$ is a unique product of two words of length $d$ by splitting the word in the middle.    It also follows that $$\Vert p\Vert_{\underline v,\rm sos}=\sum_{\vert\nu\vert=d} p_{\nu^*\nu}$$ holds for every $p\in \sum  \mathbb C\langle \underline z\rangle_{d}^2,$ so this norm is directly computable from the coefficients of $p$. 
Furthermore,  $\Vert G\Vert_{p\to\mathbb C\langle\underline z\rangle}=1,$ whenever we equip $\mathbb C\langle\underline z\rangle_{2d}$ with the norm inherited from the Schatten $p$-norm  on matrices. Consequently, Theorem \ref{mainthm} provides a dimension-independent approximation result.

 However, the only Schatten $p$-norm that gives rise to a natural norm on polynomials is the 
 Schatten $2$-norm, which in fact induces the $2$-norm of coefficients on $\mathbb C\langle \underline z\rangle$, i.e., if for $p=\sum_\omega p_\omega \omega$ we set $$\Vert p\Vert_2\coloneqq \left(\sum_{\omega} \vert p_\omega\vert^2\right)^{1/2}.$$

\begin{theorem}\label{thm:free}
Let $p\in \sum  \mathbb C\langle \underline z\rangle_{d}^2$ and let $\varepsilon >0$ be fixed. Then the $\varepsilon$-ball around $p$ with respect to $\Vert \cdot \Vert_2$ contains a sum of at most  $$\frac{\left(\sum_{\vert\nu\vert=d} p_{\nu^*\nu}\right)^2}{\varepsilon^2}$$ many Hermitian squares of elements from $\mathbb C\langle \underline z\rangle_{d}$.\end{theorem}

Note that given some $p\in \sum  \mathbb C\langle \underline z\rangle_{d}^2$, an approximating sum of squares with small Pythagoras number attains an explicit construction. First, the (only) positive semidefinite Gram matrix $M$ of $p$ can be read off directly from the coefficients of $p$. Next, we apply the construction in the proof of Proposition  \ref{prop:sanderandinfinity} to obtain an approximation $M'$ of $M$, and finally apply the Gram map again.

\subsection{Commutative Polynomials}\label{ss:poly}
Let $S$ be a compact Hausdorff space and $C(S,\mathbb C)$ the commutative $*$-algebra of complex-valued continuous functions on $S$, equipped with the usual sup-norm $\Vert\cdot\Vert_\infty$. The following lemma shows how we can bound $\Vert G_{\underline v}\Vert_{\infty\to  C(S,\mathbb C)},$ which is the largest among all operator norms, with a very natural choice of basis for polynomials.

\begin{lemma}\label{lem:op}
($i$) Let $\mathcal V\subseteq C(S,\mathbb C)$ be a subspace with a basis $\underline v=(v_1,\ldots, v_d)$ such that $$\Vert(v_1(s),\ldots,v_d(s))\Vert \leqslant 1$$ for all $s\in S$ (here we use the standard Euclidean norm on $\mathbb C^d$). Then for the corresponding Gram map $G_{\underline v}\colon{\rm Mat}_d(\mathbb C)\to C(S,\mathbb C)$ we have $$  \Vert G_{\underline v}\Vert_{\infty\to  C(S,\mathbb C)} \leqslant 1.$$
	
\noindent ($ii$) Let $S^{n-1}\subseteq \mathbb R^{n}$ be the real unit sphere and $\mathfrak m_d$ the tuple of all monomials in $x_1,\ldots, x_n$ of degree $d$. Then for every $s\in S^{n-1}$ and all $d\geqslant 1$ we have $\Vert\mathfrak m_d(s)\Vert\leqslant 1.$

\end{lemma}
\begin{proof}
($i$) For $M\in{\rm Mat}_d(\mathbb C)$ and $\underline v(s)=\left(v_1(s),\ldots, v_d(s)\right)^t$ we have \begin{align*}\Vert G_{\underline v}(M)\Vert_\infty&=\max_{s\in S}\left\vert \underline v(s)^*M \underline v(s)\right\vert \\ &\leqslant \max_{y\in\mathbb C^d, y^*y\leqslant 1} \left\vert y^*My\right\vert \\ &= \max_{y\in\mathbb C^d, y^*y=1} \left\vert y^*My\right\vert\\ &\leqslant \Vert M\Vert_\infty.\end{align*} This proves the claim.

($ii$) For $s\in S^{n-1}$ and $d\geqslant 1$ we have 
$$\Vert \mathfrak m_{d}(s)\Vert^2=\sum_{\vert\alpha\vert=d} s^{2\alpha}\leqslant \left( \sum_{i=1}^n s_i^2\right)^d=\Vert s\Vert^{2d}=1.$$
The inequality is due to the fact that each term from the sum on the left appears at least once in the sum on the right.
\end{proof}

When homogeneous polynomials are considered as functions on $S^{n-1}$, the sup-norm $\Vert\cdot\Vert_\infty$ indeed defines a norm on $\mathbb C[\underline x]_{2d}$.

\begin{theorem}\label{thm:poly}
Let $p\in\sum \mathbb C[\underline x]_d^2$ and let $\varepsilon >0$ be fixed. Then the $\varepsilon$-ball around $p$ with respect to $\Vert \cdot \Vert_\infty$ contains a sum of at most
$$ \frac{\Vert p\Vert_{\mathfrak m_d,\rm sos}}{\varepsilon}$$ many Hermitian squares from   $\mathbb C[\underline x]_d.$
\end{theorem}
\begin{proof}
Clear from Theorem \ref{mainthm} and Lemma \ref{lem:op}.\end{proof}

\begin{example}\label{ex:poly}
Consider the family of polynomials $$p_{n,d}:=\Vert\mathfrak m_d(\underline x)\Vert^2=\sum_{\vert\alpha\vert =d}\underline x^{2\alpha}\in\sum\mathbb C[\underline x]^2_d.$$  
To the best of our knowledge, the exact Pythagoras numbers of the $p_{n,d}$ are unknown. Moreover, no upper bounds that are better than the general ones seem to exist.

The identity matrix  is clearly a psd Gram matrix of $p_{n,d}$, but with full rank. So the best known upper bound to $\mathfrak p(p_{n,d})$ is   
$$\sqrt{{{2d+n-1}\choose{n-1}}}$$ from Example \ref{ex:spacedim}. 

Considering approximations of $p_{n,d}$ with respect to the $\infty$-norm on the sphere, note that $\Vert p_{n,d}\Vert_{\infty}=1$ for all $n,d$ ($\leqslant$ is Lemma \ref{lem:op} ($ii$) and $\geqslant$ is obvious by evaluating at $(1,0,\ldots 0)\in S^{n-1}$). This normalization condition allows for comparing the approximation results among different values of $n$ and $d$.

The trace of the identity matrix is ${{d+n-1}\choose{n-1}}$, but it turns out that the sos-norm of $p_{n,d}$ is actually significantly smaller. We have used a semidefinite programming solver with Python to compute $\Vert p_{n,d}\Vert_{\mathfrak m_d,\rm sos}$ for $n=3,4,5,6$ and several values of $d$. Figure \ref{fig:sosnorm} shows that it grows much slower than the general upper bound. 

Therefore, Theorem \ref{thm:poly} implies  the existence of approximations of $p_{n,d}$ from $\sum \mathbb C[\underline x]_d^2$ that have a significantly smaller Pythagoras number than the general known upper bounds, for fixed approximation error $\varepsilon > 0$.

\begin{figure}[h!]
\includegraphics[scale=0.45]{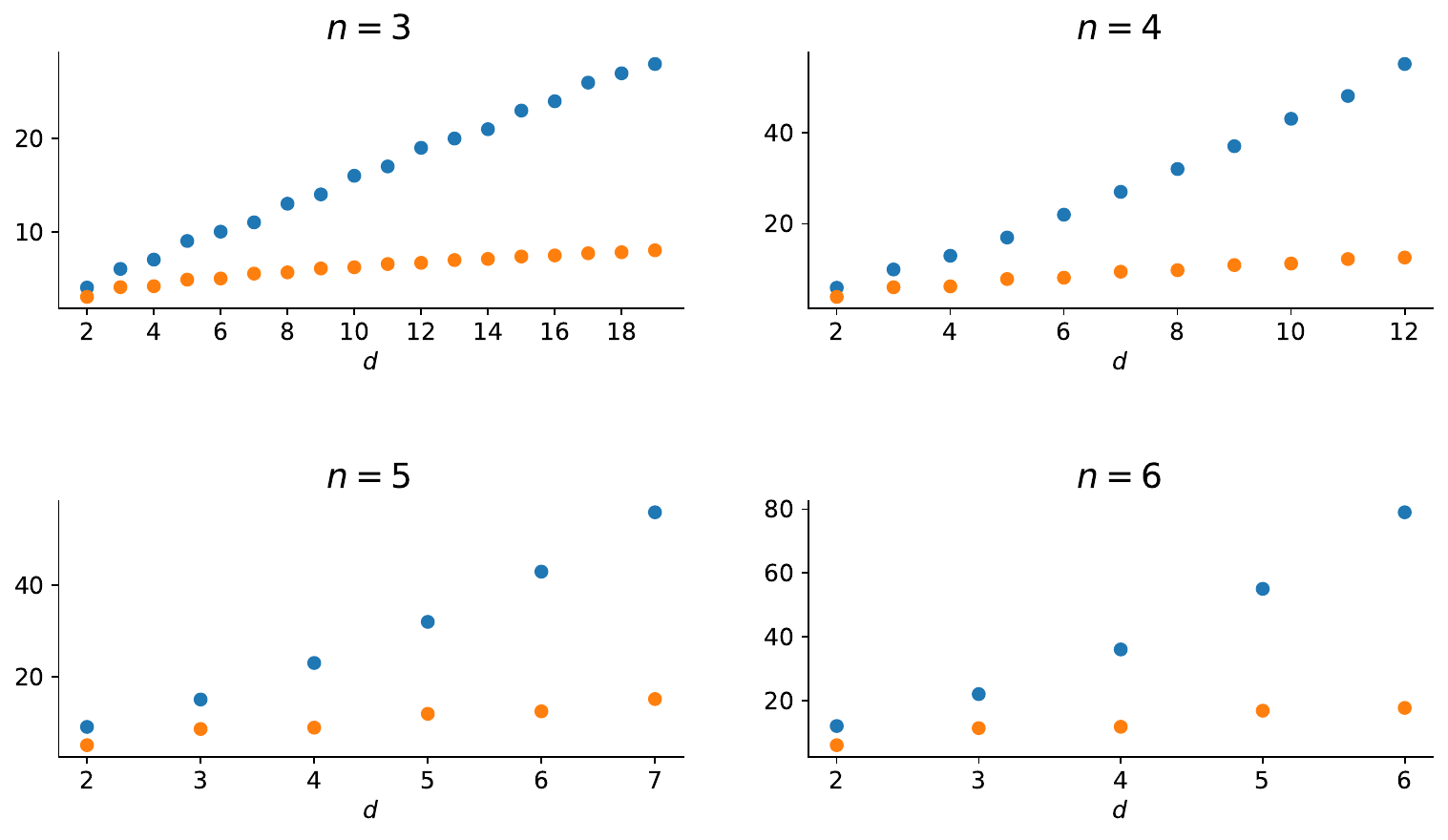}
\caption{Comparison of $\Vert p_{n,d}\Vert_{\mathfrak m_d,\rm sos}$ (orange) and the general upper bound $\sqrt{{{2d+n-1}\choose{n-1}}}$ (blue) for the Pythagoras number of $\mathbb C[\underline x]_{2d},$ for different $n$ and $d$.}
\label{fig:sosnorm}
\end{figure}

Beyond these numerical results we were not able to give asymptotic bounds of $\Vert p_{n,d}\Vert_{\mathfrak m_d,\rm sos}$. Such bounds would for example need an explicit sum of squares decomposition of $p_{n,d}$, beyond the obvious one.

\end{example}

\subsection{Some Remarks  on the SOS-Norm}\label{ss:norm} Finally, let us examine the sos-norm a little closer. 
We consider  the same general setup as in the proof of Corollary \ref{cor:impcar}, i.e.\ we fix bases $\underline v$ of $\mathcal V$, $\underline w$ of $\mathcal V^*\mathcal V$,  express  $v_i^*v_j=\sum_{l=1}^{k}\lambda_{l}^{i,j}\omega_l$, and set $$A_{l}\coloneqq \left(\overline\lambda_{l}^{i,j}\right)_{i,j}\in{\rm Mat}_d(\mathbb C)$$ for $l=1,\ldots, k$. Then for $a=\sum_{l=1}^{k}\lambda_{l}\omega_{l}\in\mathcal V^*\mathcal V$ we have $a\in\sum\mathcal V^2$ if and only if there exist some $M\in {\rm Psd}_d(\mathbb C)$ fulfilling $${\rm tr}(A_lM)=\lambda_{l} \quad \mbox{ for } l=1,\ldots, k,$$ and $\Vert a\Vert_{\underline v, \rm sos}$ is the minimal value of ${\rm tr}(M)$ among all such $M$.   Computing this is clearly a semidefinite program (in primal form). It is easy to check that the dual problem takes the following form:
 \begin{align*} (D) \quad \sup  \quad & \varphi(a)  \\
 {\rm s.t.}\quad &\varphi\colon \mathcal V^*\mathcal V\to \mathbb C \quad *\mbox{-linear}\\
 &\varphi(v^*v)\leqslant \Vert v\Vert^2 \mbox{ for all } v\in\mathcal V,
 \end{align*}
where $\Vert v\Vert$ denotes the $2$-norm of coefficients with respect to the basis $\underline v$. The duality theory of semidefinite programming immediately implies:

\begin{corollary} $\Vert a\Vert_{\underline v, \rm sos}$ is bounded from below by the optimal value of (D). 
If $a$ admits a positive definite Gram matrix, then $\Vert a\Vert_{\underline v, \rm sos}$ equals  the optimal value of (D). 
\end{corollary}

\begin{corollary}\label{cor:dual}
For  $p\in\sum\mathbb C[\underline x]_{d}^2$ we have $$\Vert p\Vert_{\infty}\leqslant\Vert p\Vert_{\mathfrak m_d, \rm sos}.$$ 
\end{corollary}
\begin{proof}
For $s\in S^{n-1}$ we consider the $*$-linear evaluation map $$\varphi_s\colon \mathbb C[\underline x]_{2d}\to\mathbb C; q\mapsto q(s).$$ We then have 
\begin{align*}
\varphi_s(q^*q)&= \vert q(s)\vert^2 \\ &\leqslant \Vert q\Vert^2 \cdot \Vert \mathfrak m_d(s)\Vert^2 \\&\leqslant \Vert q\Vert^2,
\end{align*} where the first inequality is Cauchy-Schwarz, and the second Lemma \ref{lem:op} ($ii$). Here, $\Vert q\Vert$ denotes the $2$-norm of coefficients of $q$, with respect to the basis $\mathfrak m_d$. So $\varphi_s$ is feasible for the dual problem ($D$) above, which implies $$\vert p(s)\vert= p(s)=\varphi_s(p)\leqslant \Vert p\Vert_{\mathfrak m_d,\rm sos}.$$ This proves the claim.
\end{proof}
Note that the inequality from \Cref{cor:dual}  can be strict. For the polynomials $p_{n,d}$ from Example \ref{ex:poly} we have $\Vert p_{n,d}\Vert_{\infty}=1,$ whereas the sos-norm is larger.

\bibliographystyle{abbrv}
\bibliography{references}

\newpage
\appendix
\section{Notations and norms}\label{app}

In this section, we summarize all norms and notations used throughout the paper. We have:
\begin{itemize}
	\item $\mathcal{A}$ is a $*$-algebra over $\mathbb{C}$.
	\item $\mathbb{C}[\underline{x}]_d$ is the space of homogeneous commutative polynomials of degree $d$ with complex coefficients. 
	\item $\mathbb{C}\langle \underline{z}\rangle_d$ is the space of homogeneous non-commutative polynomials of degree $d$ with complex coefficients. 
	\item ${\rm Mat}_d(\mathbb{C})$ is the space of $d \times d$ complex matrices, ${\rm Psd}_d(\mathbb C)$ the convex cone of positive semidefinite matrices.
	\item $\mathcal{L}({\rm Mat}_d(\mathbb{C}), \mathcal{A})$ is the space of linear maps from ${\rm Mat}_d(\mathbb{C})$ to $\mathcal{A}$
\end{itemize}

\vspace*{0.8cm}

\newcolumntype{C}[1]{%
 >{\vbox to 7ex\bgroup\vfill\centering}%
 p{#1}%
 <{\egroup}}
 
 \newcolumntype{R}[1]{%
 >{\vbox to 7ex\bgroup\vfill\raggedleft}%
 p{#1}%
 <{\egroup}}
 
 \newcolumntype{L}[1]{%
 >{\vbox to 7ex\bgroup\vfill\raggedright}%
 p{#1}%
 <{\egroup}}

\setlength{\tabcolsep}{2pt}

\scriptsize{
\renewcommand{\arraystretch}{0.8}
\begin{tabular}{R{1.45cm}| C{1.5cm}| C{2cm}| R{1.6cm} L{5cm}}
\multicolumn{1}{c|}{} & \multicolumn{1}{c|}{{\normalsize Symbol}} & \multicolumn{1}{c|}{{\normalsize Domain}} & \multicolumn{2}{c}{{\normalsize Definition}} \tabularnewline
\Xhline{1.5pt}
$\mathcal{A}$-norm \hspace{0.2cm} & $\Vert \cdot \Vert_\mathcal{A}$ & $\mathcal{A}$ & \multicolumn{2}{C{6.3cm}}{arbitrary vector space  norm}\tabularnewline
\hline
\makecell[r]{Euclidean \hspace{0.2cm}\\ norm \hspace{0.2cm}} & $\Vert \cdot \Vert$ & $\mathbb{C}^d$ & $\Vert v \Vert$ & $\displaystyle \coloneqq \left(\sum_{i=1}^{d} |v_i|^2\right)^{1/2}$\tabularnewline[-0.1cm]
\hline
$\infty$-norm \hspace{0.2cm} & $\Vert \cdot \Vert_{\infty}$ & $\mathbb{C}[\underline{x}]_d$ & $\displaystyle \Vert p \Vert_{\infty}$ & $\displaystyle \coloneqq \sup_{x \in S^{n-1}} |p(x_1, \ldots, x_n)|$\tabularnewline
\hline
\makecell[r]{coefficient \hspace{0.2cm}\\ $2$-norm \hspace{0.2cm}} & $\Vert \cdot \Vert_2$ & $\mathbb{C}\langle \underline{z}\rangle_d$ & $\displaystyle \left\Vert \sum_{\omega} p_{\omega} \omega \right\Vert_{2}$ & $\displaystyle \coloneqq \left(\sum_{\omega} |p_{\omega}|^2 \right)^{1/2}$\tabularnewline
\hline
\makecell[r]{Schatten \hspace{0.2cm}\\ $p$-norm \hspace{0.2cm}} & $\Vert \cdot \Vert_p$ & ${\rm Mat}_d(\mathbb{C})$ & $\displaystyle \Vert M \Vert_{p}$ & $\coloneqq \Big({{\rm tr}\big(\sqrt{M^*M}^p\big)}\Big)^{1/p}$\tabularnewline
\hline
\makecell[r]{Schatten \hspace{0.2cm}\\ $\infty$-norm \hspace{0.2cm}} & $\Vert \cdot \Vert_\infty$ & ${\rm Mat}_d(\mathbb{C})$ & $\displaystyle \Vert M \Vert_{\infty}$ & $\coloneqq \sigma_{\max}(M), \mbox{ largest singular value}$\tabularnewline
\hline

\makecell[r]{$p \to \mathcal{A}$ \hspace{0.2cm}\\ norm \hspace{0.2cm}} & $\Vert \cdot \Vert_{p \to \mathcal{A}}$ & $\mathcal{L}({\rm Mat}_d(\mathbb{C}), \mathcal{A})$ & $\Vert G \Vert_{p \to \mathcal{A}}$ & $\coloneqq \displaystyle \max_{M \in {\rm Mat}_d} \frac{\Vert G(M) \Vert_{\mathcal{A}}}{\Vert M \Vert_{p}}$\tabularnewline
\hline
sos-norm \hspace{0.2cm} & $\Vert \cdot \Vert_{\underline{v}, \rm sos}$ & $\sum \mathcal{V}^2\subseteq \mathcal A$ & $\displaystyle \Vert a \Vert_{\underline{v}, \rm sos}$ & $\displaystyle \coloneqq \min \Big\{\textrm{tr}(M): M \textrm{ psd, }  G_{\underline{v}}(M) = a \Big\}$\tabularnewline
\end{tabular}}
\normalsize

\end{document}